\documentclass[12pt]{article}
\usepackage{amsmath,amssymb,amsthm, comment}

\newtheorem{theorem}{Theorem}[section]
\newtheorem{thmy}{Theorem}

\newtheorem{lemma}[theorem]{Lemma}

\def\barr{\begin{array}}
\def\earr{\end{array}}
\title{Arithmetic progressions in finite groups}
\author{Marius T\u arn\u auceanu}
\date{March 23, 2020}

\begin{document}

\maketitle

\begin{abstract}
In this paper, we describe the structure of finite groups whose element orders or proper (abelian) subgroup orders form an arithmetic progression of ratio $r\geq 2$. This extends the case $r=1$ studied in previous papers \cite{1,8,4}.
\end{abstract}

{\small
\noindent
{\bf MSC2000\,:} Primary 20D60; Secondary 20D99.

\noindent
{\bf Key words\,:} finite groups, ${\rm CP}_1$-groups, element orders, subgroup orders, arithmetic progressions.}
\vspace{-2mm}

\section{Introduction}

Given a finite group $G$, we denote by $\pi_e(G)$ the set of element orders of $G$, by $\pi_s(G)$ the set of proper subgroup orders of $G$ and by $\pi_{as}(G)$ the set of proper abelian subgroup orders of $G$. The starting point for our discussion is given by \cite{1,8,4} that classify finite groups $G$ in which the elements of the above sets are consecutive integers, i.e. arithmetic progressions of ratio $r=1$.

\begin{thmy}
If $\pi_e(G)=\{1,2,\dots,n\}$, then $n\leq 8$ and one of the following holds:
\begin{itemize}
\item[{\rm a)}] $n\leq 2$ and $G$ is an elementary abelian $2$-group;
\item[{\rm b)}] $n=3$ and $G=[N]Q$ is a Frobenius group, where either $N\cong C_3^t$, $Q\cong C_2$ or $N\cong C_2^{2t}$, $Q\cong C_3$;
\item[{\rm c)}] $n=4$ and $G=[N]Q$ and one of the following holds:
\begin{itemize}
\item[{\rm (i)}] $N$ has exponent $4$ and class $\leq 2$ and $Q\cong C_3$;
\item[{\rm (ii)}] $N=C_2^{2t}$ and $Q\cong S_3$;
\item[{\rm (iii)}] $N=C_3^{2t}$ and $Q\cong S_3$ or $Q_8$ and $G$ is a Frobenius group;\newpage
\end{itemize}
\item[{\rm d)}] $n=5$ and $G\cong A_6$ or $G=[N]Q$, where $Q\cong A_5$ and $N$ is an elementary abelian $2$-group and a direct sum of natural $SL(2,4)$-modules;
\item[{\rm e)}] $n=6$ and $G$ is one of the following types:
\begin{itemize}
\item[{\rm (i)}] $G=[P_5]Q$ is a Frobenius group, where $Q\cong [C_3]C_4$ or $Q\cong SL(2,3)$ and $P_5=C_5^{2t}$;
\item[{\rm (ii)}] $G/O_2(G)\cong A_5$ and $O_2(G)$ is elementary abelian and a direct sum of natural and orthogonal $SL(2,4)$-modules;
\item[{\rm (iii)}] $G=S_5$ or $G=S_6$;
\end{itemize}
\item[{\rm f)}] $n=7$ and $G\cong A_7$;
\item[{\rm g)}] $n=8$ and $G=[PSL(3,4)]\langle\beta\rangle$, where $\beta$ is a unitary automorphism of $PSL(3,4)$.
\end{itemize}
\end{thmy}

\begin{thmy}
If $\pi_s(G)=\{1,2,\dots,n\}$, then $n\leq 4$ and one of the following holds:
\begin{itemize}
\item[{\rm a)}] $n=1$ and $G\cong C_p$, $p$ a prime;
\item[{\rm b)}] $n=2$ and $G\cong C_2\times C_2$ or $G\cong C_4$;
\item[{\rm c)}] $n=3$ and $G\cong S_3$ or $G\cong C_6$;
\item[{\rm d)}] $n=4$ and $G\cong A_4$.
\end{itemize}
\end{thmy}

\begin{thmy}
If $\pi_{as}(G)=\{1,2,\dots,n\}$, then $n\leq 6$ and one of the following holds:
\begin{itemize}
\item[{\rm a)}] $n=1$ and $G\cong C_p$, $p$ a prime;
\item[{\rm b)}] $n=2$ and $G\cong C_2\times C_2$ or $G\cong C_4$;
\item[{\rm c)}] $n=3$ and $G\cong S_3$ or $G\cong C_6$;
\item[{\rm d)}] $n=4$ and $G\cong A_4$ or $G\cong S_4$;
\item[{\rm e)}] $n=4$ and $G\cong A_5$;
\item[{\rm f)}] $n=4$ and $G\cong S_5$.
\end{itemize}
\end{thmy}
\smallskip

The current paper deals with the case $r\geq 2$. Then we may assume that $G$ is of odd order, by Cauchy's theorem. Our main results are as follow.\newpage

\begin{theorem}
Let $G$ be a finite group of odd order. Then $\pi_e(G)$ is an arithmetic progression if and only if $G$ is either a $p$-group of exponent $p$ or a Frobenius group with kernel $P\in Syl_p(G)$ of exponent $p$ and complement $Q\in Syl_q(G)$ of order $q$, and moreover $p=2q-1$ or $q=2p-1$.
\end{theorem}
\smallskip

It is not known whether there are infinitely many pairs of primes $(q,p)$ with $p=2q-1$.\footnote{These pairs are known as \textit{Cunningham chains of the second kind of length $2$} - see, e.g., the sequence A005382 in \cite{9}.} So, we cannot decide whether there are infinitely many groups of the above second kind. Note that the smallest one is of order $75$, namely ${\rm SmallGroup}(75,2)=[P]Q$, where $P\cong C_5\times C_5$ and $Q\cong C_3$.

\begin{theorem}
Let $G$ be a finite group of odd order. Then $\pi_s(G)$ or $\pi_{as}(G)$ is an arithmetic progression if and only if either $G\cong C_p\times C_p$ or $G$ is cyclic of one of the following types: $C_{p^i}$, $i=1,2$, or $C_{pq}$ with $p$ and $q$ distinct primes such that $p=2q-1$ or $q=2p-1$.
\end{theorem}
\smallskip

Together with Theorems A, B and C, these lead to a complete classification of finite groups $G$ for which $\pi_e(G)$, $\pi_s(G)$ or $\pi_{as}(G)$ are arithmetic progressions.

For the proof of Theorems 1.1 and 1.2 we need the classification of finite groups with all non-trivial elements of prime order (see \cite{3,2}) and the classification of finite minimal non-cyclic groups (see \cite{7}).

\begin{theorem}
Let $G$ be a finite groups having all non-trivial elements of prime order. Then:
\begin{itemize}
\item[{\rm a)}] $G$ is nilpotent if and only if $G$ is a $p$-group of exponent $p$.
\item[{\rm b)}] $G$ is solvable and non-nilpotent if and only if $G$ is a Frobenius group with kernel $P\in Syl_p(G)$, with $P$ a $p$-group of exponent $p$ and complement $Q\in Syl_q(G)$, with $|Q|=q$. Moreover, if $|G|=p^nq$ then $G$ has a chief series $G=G_0>P=G_1>G_2>\dots>G_k>G_{k+1}=1$ such that for every $1\leq i\leq k$ one has $G_i/G_{i+1}\leq Z(G/G_{i+1})$, $Q$ acts irreducibly on $G_i/G_{i+1}$ and $|G_i/G_{i+1}|=p^b$, where $b$ is the exponent of $p\,\, {\rm (mod}\,\, q{\rm )}$.
\item[{\rm c)}] $G$ is non-solvable if and only if $G\cong A_5$.
\end{itemize}
\end{theorem}

\begin{theorem}
A finite group $G$ is a minimal non-cyclic group if and only if it is isomorphic to one of the following groups:\vspace{2mm}
\begin{itemize}
\item[{\rm a)}] $C_p\times C_p$, where $p$ is a prime;
\item[{\rm b)}] $Q_8$;
\item[{\rm c)}] $\langle a,b \mid a^p=b^{q^m}=1, b^{-1}ab=a^r \rangle$, where $p$ and $q$ are distinct primes and $r\not\equiv 1\, {\rm (mod}\, p{\rm )}$, $r^q\equiv 1\, {\rm (mod}\, p{\rm )}$.
\end{itemize}
\end{theorem}

We also need the following result of \cite{6}, known as Lucido’s Three Primes Lemma:

\begin{lemma}
Let $G$ be a finite solvable group. If $p$, $q$, $r$ are distinct primes dividing $|G|$, then $G$ contains an element of order the product of two of these three primes.
\end{lemma}

Finally, we formulate an open problem related to Theorem 1.2.

\medskip\noindent{\bf Open problem.} Determine all finite groups for which the set of proper normal subgroup orders is an arithmetic progression.
\medskip

Most of our notation is standard and will usually not be repeated here. Elementary notions and results on groups can be found in \cite{5}.

\section{Proof of Theorem 1.1}

We start with the following lemma.

\begin{lemma}
Let $G$ be a finite group of odd order such that $\pi_e(G)$ is an arithmetic progression. Then all non-trivial elements of $G$ are of prime order.
\end{lemma}

\begin{proof}
Let $\pi_e(G)=\{a_0=1,a_1,\dots,a_n\}$. Then $a_1$ is the smallest prime $p$ dividing $|G|$ and so $a_i=ip-i+1$, for all $i$. In particular, we have $a_{p+1}=p^2$. Since for every $d\in\pi_e(G)$ the set of divisors of $d$ is contained in $\pi_e(G)$, we infer that $a_2$, $a_3$, ..., $a_p$ are also primes. We distinguish the following two cases:

\smallskip
\hspace{10mm}\noindent{\bf Case 1.} $p=3$

\noindent Then $\pi_e(G)=\{1,3,5,7,\dots,m\}$.

If $m\geq 41$ then $\left[\frac{m}{8}\right]\geq 2$, and thus there are three primes $p_1$, $p_2$ and $p_3$ such that
\begin{equation}
\left[\frac{m}{8}\right]<p_1<2\left[\frac{m}{8}\right]<p_2<4\left[\frac{m}{8}\right]<p_3<8\left[\frac{m}{8}\right]\leq m,\nonumber
\end{equation}by the well-known Bertrand's postulate. Since $G$ is solvable, Lemma 1.5 shows that there exists $a\in G$ with $o(a)\in\{p_1p_2,p_2p_3,p_3p_1\}$, which leads to\newpage
\begin{equation}
o(a)\geq p_1p_2>2\left[\frac{m}{8}\right]^2>m,\nonumber
\end{equation}a contradiction.

The cases $m\in\{7,9,\dots,39\}$ can be eliminated directly by using Lemma 1.5. For example, for $m=39$ we have $29,31,37\in\pi_e(G)$ and so there exists $a\in G$ with $o(a)\in\{29\cdot 31,31\cdot 37,37\cdot 29\}$, that is $o(a)>39$, a contradiction.

Consequently, $m\leq 5$, i.e. $\pi_e(G)\subseteq\{1,3,5\}$, showing that all non-trivial elements of $G$ are of prime order.

\smallskip
\hspace{10mm}\noindent{\bf Case 2.} $p\geq 5$

\noindent Let $q$ be the smallest prime not dividing $p-1$. Then $q<p$. On the other hand, it is well-known that an arithmetic progression of ratio $r$ cannot contain more consecutive prime terms than the value of the smallest prime that does not divide $r$. This shows that $p\leq q$. Thus, in this case we get a contradiction, completing the proof.
\end{proof}

We are now able to prove Theorem 1.1.

\medskip\noindent{\bf Proof of Theorem 1.1.} Let $G$ be a finite group of odd order such that $\pi_e(G)$ is an arithmetic progression. Then all non-trivial elements of $G$ are of prime order by Lemma 2.1, and therefore $G$ is either a $p$-group of exponent $p$ or a Frobenius group with kernel $P\in Syl_p(G)$ of exponent $p$ and complement $Q\in Syl_q(G)$ of order $q$ by Theorem 1.3. Clearly, in the second case we have $\pi_e(G)=\{1,p,q\}$, implying that $p=2q-1$ or $q=2p-1$, as desired.\qed

\section{Proof of Theorem 1.2}

We first present a lemma, whose proof is similar to that of Lemma 2.1.

\begin{lemma}
Let $G$ be a finite group of odd order such that $\pi_s(G)$ or $\pi_{as}(G)$ is an arithmetic progression. Then all non-trivial proper subgroups of $G$ are of prime order.
\end{lemma}

This easily leads to a proof of Theorem 1.2.

\medskip\noindent{\bf Proof of Theorem 1.2.}
Let $G$ be a finite group of odd order such that $\pi_s(G)$ or $\pi_{as}(G)$ is an arithmetic progression. Then all non-trivial proper subgroups of $G$ are cyclic by Lemma 3.1. It follows that $G$ is either cyclic or minimal non-cyclic. In the first case we get $G\cong C_{p^i}$, $i=1,2$, or $G\cong C_{pq}$ with $p$ and $q$ distinct primes such that $p=2q-1$ or $q=2p-1$, while in the\newpage \noindent second one we get $G\cong C_p\times C_p$ by Theorem 1.4 (note that we cannot have $G\cong \langle a,b \mid a^p=b^{q^m}=1, b^{-1}ab=a^r \rangle$ because Lemma 3.1 would imply $m=1$, i.e. $G$ would be a non-abelian group of order $pq$, and so $\pi_s(G)=\pi_{as}(G)=\{1,p,q\}$; assuming $p<q$, we find $p\mid q-1$ and $q=2p-1$, leading to $p\mid 2p-2$, a contradiction). This completes the proof.\qed

\vspace*{3ex}\small

\hfill
\begin{minipage}[t]{5cm}
Marius T\u arn\u auceanu \\
Faculty of  Mathematics \\
``Al.I. Cuza'' University \\
Ia\c si, Romania \\
e-mail: {\tt tarnauc@uaic.ro}
\end{minipage}

\end{document}